\documentclass{amsart}
\usepackage{setspace}
\usepackage{a4}
\usepackage{amsthm}
\usepackage{latexsym}
\usepackage{amsfonts}
\usepackage{graphicx}
\usepackage{textcomp}
\usepackage{cite}
\usepackage{enumerate}
\usepackage{amssymb}
\usepackage{hyperref}
\usepackage{amsmath}
\usepackage{tikz}
\usepackage[mathscr]{euscript}
\usepackage{mathtools}
\newtheorem{theorem}{Theorem}[section]

\newtheorem{corollary}[theorem] {Corollary}
\newtheorem{definition}[theorem]{Definition}
\newtheorem{example}[theorem]{Example}

\newtheorem{problem}[theorem]{Problem}

\title{This is the title}
\usepackage{fancyhdr}

\begin{document}
\hrule\hrule\hrule\hrule\hrule
\vspace{0.3cm}	
\begin{center}
{\bf{Noncommutative Spherical Codes}}\\
\vspace{0.3cm}
\hrule\hrule\hrule\hrule\hrule
\vspace{0.3cm}
\textbf{K. Mahesh Krishna}\\
School of Mathematics and Natural Sciences\\
Chanakya University Global Campus\\
NH-648, Haraluru Village\\
Devanahalli Taluk, 	Bengaluru  North District\\
Karnataka State, 562 110, India \\
Email: kmaheshak@gmail.com

Date: \today

\end{center}

\hrule\hrule
\vspace{0.5cm}

%--------------------------------------
\textbf{Abstract}: Spherical codes, with a rich history spanning nearly five centuries, remain an area of active mathematical exploration and are far from being fully understood. These codes, which arise naturally in problems of geometry, combinatorics, and information theory, continue to challenge researchers with their intricate structure and unresolved questions. Inspired by Polya’s heuristic principle of “vary the problem,” we extend the classical framework by introducing the notion of noncommutative spherical codes, with particular emphasis on the noncommutative Newton–Gregory kissing number problem. This generalization moves beyond the traditional Euclidean setting into the realm of operator algebras and Hilbert C*-modules, thereby opening new avenues of investigation. A cornerstone in the study of spherical codes is the celebrated Delsarte–Goethals–Seidel–Kabatianskii–Levenshtein linear programming bound, developed over the past half-century. This bound employs Gegenbauer polynomials to establish sharp upper limits on the size of spherical codes, and it has served as a fundamental tool in coding theory and discrete geometry. Remarkably, a recent elegant one-line proof by Pfender [\textit{J. Combin. Theory Ser. A, 2007}] provides a streamlined derivation of a variant of this bound. We demonstrate that Pfender’s argument can be extended naturally to the setting of Hilbert C*-modules, thereby enriching the theory with noncommutative analogues.



\hrule

%\tableofcontents
\hrule
\section{Introduction}
A finite set of points distributed on the surface of a unit sphere in a Euclidean space, and satisfying specific criteria, is a concept with profound implications across numerous scientific disciplines, including Mathematics, Physics, Chemistry, Engineering, and Biology. This topic represents one of the most enduring areas of scientific inquiry.
Several classic problems focus on the study of such finite point arrangements on a sphere. Few are given below. 
\begin{enumerate}[\upshape (i)]
	\item Tammes Problem: This problem investigates how to arrange a given number of points on a sphere such that the minimum distance between any two points is maximized.
\item Thomson Problem: Also known as Problem 7 in Steve Smale's list of ``Mathematical Problems for the Next Century," this problem seeks the configuration of a given number of points on a sphere that minimizes the total electrostatic potential energy.
\item Spherical Designs Problem: This area deals with finding sets of points that provide highly uniform distribution on the sphere, such that they accurately approximate the integration of polynomials over the sphere's surface.
\item Spherical Codes Problem: Similar to the Tammes problem, this problem aims to maximize the minimum angular separation (or distance) between points, relevant for applications in communication and coding theory.
\item Equiangular Lines Problem: This problem involves finding the maximum number of lines passing through the origin that are all pairwise separated by the same angle.
\end{enumerate}
In this article, we are concerned about spherical codes. We recall the definition.
Let $d\in \mathbb{N}$ and $\theta \in [0, 2\pi)$. A set $\{\tau_j\}_{j=1}^n$  of unit vectors  in $\mathbb{R}^d$	is said to be a \textbf{$(d,n,\theta )$-spherical code} \cite{ZONG} in $\mathbb{R}^d$ if 
 	\begin{align}\label{SCI}
 	\langle \tau_j, \tau_k\rangle\leq \cos \theta , \quad \forall 1\leq j, k \leq n, j \neq k.
 \end{align}
Since 
\begin{align*}
	\langle \tau, \omega\rangle =\frac{2-\|\tau-\omega\|^2}{2}, \quad \forall \tau, \omega \in \mathbb{R}^d, 
\end{align*}
we can rewrite Inequality (\ref{SCI}) as 
\begin{align*}\label{SCN}
	\|\tau_j-\tau_k\|\geq  \sqrt{2(1-\cos \theta)}, \quad \forall 1\leq j, k \leq n, j \neq k.	
\end{align*}	
 Fundamental problem associated with spherical codes is the following.
 \begin{problem}\label{SCP}
 	Given $d$ and $\theta$, what is the maximum $n$ such that there exists a $(d,n,\theta )$-spherical code $\{\tau_j\}_{j=1}^n$  in $\mathbb{R}^d$?
 \end{problem}
 The case $\theta=\pi/3$ is known as the famous \textbf{(Newton-Gregory) kissing number problem} (KNP). KNP  is still not completely resolved in every dimension (but resolved in dimensions $d=1$ ($n=2$), $d=2$ ($n=6$), $d=3$ ($n=12$), $d=4$ ($n=24$), $d=8$ ($n=240$), $d=24$ ($n=196560$)) \cite{ANSTREICHER, PFENDER, CASSELMAN, MUSIN1, MUSIN2, ODLYZKOSLOANE,  SCHUTTEVANDERWAERDEN,  PFENDERZIEGLER, BACHOCVELLENTIN, MITTELMANNVALLENTIN, BOYVALENKOVDODUNEKOVMUSIN, BOROCZKY, MAEHARA, GLAZYRIN, LIBERTI, LEECH, KALLALKANWANG, JENSSENJOOSPERKINS, MACHADOFILHO, KUKLIN}. We refer  \cite{DELSARTEGOETHALSSEIDEL, BOYVALENKOVDRAGNEVHARDINSTOYANOVA, BARGMUSIN, BOYVALENKOVDRAGNEVHARDINSTOYANOVA2, ERICSONZINOVIEV, SARDARIZARGAR, MUSIN3, BANNAISLOANE, BACHOCVALLENTIN2, CONWAYSLOANE, COHNJIAOKUMARTORQUATO, SAMORODNITSKY, BOYVALENKOVDANEV, BOYVALENKOVDANEVLANDGEV, SLOANE, BANNAIBANNAI, BOROCZKYGLAZYRIN}
  for more on spherical codes. Problem \ref{SCP} has connection even with important sphere packing  problem \cite{COHNZHAO}. Repeatedly  used method for obtaining upper bounds on spherical codes is the Delsarte-Goethals-Seidel-Kabatianskii-Levenshtein
bound. It is obtained using Gegenbauer polynomials. Let $n \in \mathbb{N}$ be fixed. The Gegenbauer polynomials are defined inductively as 
\begin{align*}
	G_0^{(n)}(r)&\coloneqq1,\quad \forall r \in [-1,1],\\
	G_1^{(n)}(r)&\coloneqq r,\quad \forall r \in [-1,1],\\
	&\quad\vdots\\
	G_k^{(n)}(r)&\coloneqq\frac{(2k+n-4)rG_{k-1}^{(n)}(r)-(k-1)	G_{k-2}^{(n)}(r)}{k+n-3}, \quad \forall r \in [-1,1], \quad \forall k \geq 2.
\end{align*}
Then the family $\{G_k^{(n)}\}_{k=0}^\infty$ is orthogonal on the interval $[-1,1] $ with respect to the weight 
\begin{align*}
	\rho(r)\coloneqq (1-r^2)^\frac{n-3}{2}, \quad \forall r \in [-1,1].
\end{align*}
\begin{theorem} \cite{DELSARTEGOETHALSSEIDEL, ERICSONZINOVIEV} (\textbf{Delsarte-Goethals-Seidel-Kabatianskii-Levenshtein Linear Programming Bound}) \label{DGS}
	Let $\{\tau_j\}_{j=1}^n$  be a  $(d,n,\theta )$-spherical code in $\mathbb{R}^d$. 
	Let $P$ be a real polynomial satisfying following conditions.
	\begin{enumerate}[\upshape(i)]
		\item $P(r)\leq 0$ for all $-1\leq r\leq \cos \theta$.
		\item Coefficients in the  Gegenbauer expansion 
		\begin{align*}
			P=\sum_{k=0}^{m}a_kG_k^{(n)}
		\end{align*}
	satisfy 
	\begin{align*}
a_0>0, \quad a_k\geq 0, ~\forall 1\leq k \leq m.
	\end{align*}
	\end{enumerate}
Then
\begin{align*}
	n \leq \frac{P(1)}{a_0}.
\end{align*}
\end{theorem}
 In 2007, Pfender made a breakthrough by giving a  one-line proof for a variant of Theorem  \ref{DGS}.
 \begin{theorem}  \cite{PFENDER} (\textbf{Delsarte-Goethals-Seidel-Kabatianskii-Levenshtein-Pfender Bound}) \label{PFENDERT}
 	Let $\{\tau_j\}_{j=1}^n$  be a  $(d,n,\theta )$-spherical code in $\mathbb{R}^d$. Let $c>0$ and $\phi:[-1,1]\to \mathbb{R}$ be a function satisfying following.
 	\begin{enumerate}[\upshape(i)]
 		\item 
 		\begin{align*}
 			\sum_{j=1}^n\sum_{k =1}^n\phi(\langle \tau_j, \tau_k\rangle)\geq 0.	
 		\end{align*}
 		\item $\phi(r)+c\leq 0$ for all $-1\leq r \leq \cos \theta$.
 	\end{enumerate}
 Then 
 \begin{align*}
 	n\leq \frac{\phi(1)+c}{c}.
 \end{align*}
In particular, if $\phi(1)+c\leq 1$, then $n\leq 1/c$.
 \end{theorem}
In this paper, we introduce the notion of noncommutative spherical codes. We show that Theorem  \ref{PFENDERT} has an extension to Hilbert C*-modules.

 \section{Noncommutative Spherical Codes}
 Let $\mathcal{A}$ be a unital C*-algebra. For $d \in \mathbb{N}$, let $\mathcal{A}^d$ be the standard (left) Hilbert C*-module  \cite{WEGGEOLSEN}  equipped with the inner product 
 \begin{align*}
 	\langle (a_j)_{j=1}^d,(b_j)_{j=1}^d\rangle := \sum_{j=1}^da_jb^*_j,  \quad \forall (a_j)_{j=1}^d,(b_j)_{j=1}^d \in \mathcal{A}^d
 \end{align*}
 and the norm 
 \begin{align*}
 	\|(a_j)_{j=1}^d\|:= \left\| \sum_{j=1}^da_ja^*_j\right\|^\frac{1}{2}, \quad \forall (a_j)_{j=1}^d\in 	\mathcal{A}^d.
 \end{align*}
We introduce noncommutative spherical codes  as follows. 
\begin{definition}\label{PCODEDEFINITION}
	Let $d\in \mathbb{N}$ and $\theta \in [0, 2 \pi)$. Let $\mathcal{A}$ be a unital C*-algebra. A set $\{\tau_j\}_{j=1}^n$  of vectors  in $ \mathcal{A}^d$	is said to be a \textbf{noncommutative $(d,n,\theta)$-spherical code} or  \textbf{$(d,n,\theta)$-modular code} in $\mathcal{A}^d$ if following conditions hold.
	\begin{enumerate}[\upshape(i)]
		\item $	\langle \tau_j, \tau_j\rangle=1$ for all  $1\leq j \leq n$.
		\item  
		\begin{equation}\label{NE}
		2-\langle \tau_j, \tau_k \rangle -\langle \tau_k, \tau_j \rangle=\langle \tau_j-\tau_k,  \tau_j-\tau_k\rangle\geq 2(1-\cos \theta), \quad  \forall 1\leq j, k \leq n, j \neq k.
		\end{equation}
		\end{enumerate}
		We call the  case $\theta=\pi/3$  as the \textbf{noncommutative kissing number problem}.
\end{definition}
Note that, even though Inequality (\ref{NE}) reduces to Inequality (\ref{SCI}) whenever $\mathcal{A}=\mathbb{C}$, Inequality (\ref{NE}) is challenging even for commutative C*-algebras. Reason is that we are comparing positive elements in C*-algebras which is difficult to handle than inequalities arising from norm.
\begin{example}
	Let $\{\tau_j\}_{j=1}^d$ be an orthonormal basis for $\mathcal{A}^d$. Then $\{\tau_j\}_{j=1}^n$ is $(d,d,\theta)$-modular code in $\mathcal{A}^d$ for every $ \theta \in [0, \pi/2)$. 
\end{example}
\begin{example}
	Let $\{\tau_j\}_{j=1}^n$ be any collection (in particular,  a frame) in $\mathcal{A}^d$ such that $\langle \tau_j, \tau_j\rangle =1$ for all $1\leq j \leq n$. Choose  $ \theta \in [0, 2 \pi)$  such that $\langle \tau_j-\tau_k,  \tau_j-\tau_k\rangle\geq 2(1-\cos \theta)$ for all $ 1\leq j, k \leq n, j \neq k.$
	Then $\{\tau_j\}_{j=1}^n$ is $(d, n, \theta)$-modular code in $\mathcal{A}^d$.
\end{example}
Let $\{\tau_j\}_{j=1}^n$  be a   noncommutative $(d,n,\theta )$-spherical code  in $\mathcal{A}^d$. Since square root respects the order of positive elements in a C*-algebra, we have 
\begin{align}\label{PN}
	\|\tau_j-\tau_k\|\geq \sqrt{2(1-\cos \theta)},	\quad  \forall 1\leq j, k \leq n,  j \neq k.
\end{align}
However, note that Inequality (\ref{PN}) may not give Inequality   (\ref{NE}). 
Following is the noncommutative version of Theorem \ref{PFENDERT}.
\begin{theorem} (\textbf{Noncommutative Delsarte-Goethals-Seidel-Kabatianskii-Levenshtein-Pfender Spherical Codes Bound}) \label{PDGSP}
Let $\mathcal{A}$ be a unital C*-algebra and $\mathcal{A}^+\coloneqq \{a^*a:a \in \mathcal{A}\}$ be the set of all positive elements in $\mathcal{A}$.	Let $\{\tau_j\}_{j=1}^n$  be a  noncommutative $(d,n,\theta )$-spherical code in $\mathcal{A}^d$. Let $c\in (0, \infty)$ and $\phi:\mathcal{A}^+\to \mathbb{R}$ be a function satisfying following.
	\begin{enumerate}[\upshape(i)]
		\item $\sum_{1\leq j,k \leq n}\phi(\langle \tau_j-\tau_k, \tau_j-\tau_k\rangle)\geq 0$.
		\item $\phi(a)+c\leq 0$ for all $ a\in \mathcal{A}^+ $  with  $ a \geq 2(1-\cos \theta)$.
	\end{enumerate}
Then 
	\begin{align*}
		n\leq \frac{\phi(0)+c}{c}.
	\end{align*}
	In particular, if $\phi(0)+c\leq 1$, then $n\leq 1/c$.	
\end{theorem}
\begin{proof}
	Define $\psi:\mathcal{A}^+\ni a \mapsto \psi(a)\coloneqq \phi(a)+c\in \mathbb{R}$. Then 
	\begin{align*}
		\sum_{1\leq j,k \leq n}\psi(\langle \tau_j-\tau_k, \tau_j-\tau_k\rangle)&=\sum_{j=1}^{n}\psi(0)+\sum_{1\leq j,k \leq n, j \neq k}\psi(\langle \tau_j-\tau_k, \tau_j-\tau_k\rangle)\\
		&=n(\phi(0)+c)+\sum_{1\leq j,k \leq n, j \neq k}(\phi(\langle \tau_j-\tau_k, \tau_j-\tau_k\rangle)+c)\\
		&\leq n(\phi(0)+c)+0=n(\phi(0)+c).
	\end{align*}
	We also have 
	\begin{align*}
		\sum_{1\leq j,k \leq n}\psi(\langle \tau_j-\tau_k, \tau_j-\tau_k\rangle)=\sum_{1\leq j,k \leq n}(\phi(\langle \tau_j-\tau_k, \tau_j-\tau_k\rangle)+c)=\sum_{1\leq j,k \leq n}\phi(\langle \tau_j-\tau_k, \tau_j-\tau_k\rangle)+cn^2.
	\end{align*}
	Therefore 
	\begin{align*}
		cn^2\leq \sum_{1\leq j,k \leq n}\phi(\langle \tau_j-\tau_k, \tau_j-\tau_k\rangle)+cn^2=\sum_{1\leq j,k \leq n}\psi(\langle \tau_j-\tau_k, \tau_j-\tau_k\rangle)\leq n(\phi(0)+c).
	\end{align*}
\end{proof}
\begin{corollary} (\textbf{Noncommutative Delsarte-Goethals-Seidel-Kabatianskii-Levenshtein-Pfender Kissing Number Bound})
Let $\{\tau_j\}_{j=1}^n$  be a  noncommutative $(d,n,\pi/3 )$-spherical code in $\mathcal{A}^d$. Let $c\in (0, \infty)$ and $\phi:\mathcal{A}^+\to \mathbb{R}$ be a function satisfying following.
\begin{enumerate}[\upshape(i)]
	\item $\sum_{1\leq j,k \leq n}\phi(\langle \tau_j-\tau_k, \tau_j-\tau_k\rangle)\geq 0$.
	\item $\phi(a)+c\leq 0$ for all $ a\in \mathcal{A}^+ $  with  $ a \geq 1$.
\end{enumerate}
Then 
\begin{align*}
	n\leq \frac{\phi(0)+c}{c}.
\end{align*}
In particular, if $\phi(0)+c\leq 1$, then $n\leq 1/c$.		
\end{corollary}	
Following generalization of Theorem \ref{PDGSP} is clear.
\begin{theorem}
	Let $\{\tau_j\}_{j=1}^n$  be a  noncommutative $(d,n,\theta )$-spherical code in $\mathcal{A}^d$. Let $c\in (0, \infty)$ and 
	\begin{align*}
		\phi:\{\langle \tau_j-\tau_k, \tau_j-\tau_k\rangle:1\leq j, k \leq n\}\to \mathbb{R}	
	\end{align*}
	be a function satisfying following.
	\begin{enumerate}[\upshape(i)]
		\item $\sum_{1\leq j,k \leq n}\phi(\langle \tau_j-\tau_k, \tau_j-\tau_k\rangle)\geq 0$.
		\item $\phi(a)+c\leq 0$ for all $ a \in \{\langle \tau_j-\tau_k, \tau_j-\tau_k\rangle:1\leq j, k \leq n, j \neq k\}$.
	\end{enumerate}
	Then 
	\begin{align*}
		n\leq \frac{\phi(0)+c}{c}.
	\end{align*}
	In particular, if $\phi(0)+c\leq 1$, then $n\leq 1/c$.	
\end{theorem}
We remark that all results in the Gegenbauer polynomials use the real numbers properties such as commutativity, inverse of nonzero numbers, Lebesgue measure, and total order properties of real numbers. As these are not available for C*-algebras, Gegenbauer polynomials cannot be used in the noncommutative setting.

\section{Conclusion}
We made a far-reaching generalization of spherical codes by introducing noncommutative spherical codes. We showed that the Pfender bound obtained for spherical codes extends to Hilbert C*-modules. Unlike the real case, we observed that noncommutative spherical codes are extremely hard to resolve. As spherical codes themselves are not completely understood after five centuries, we believe that noncommutative spherical codes will not be completely understood even after millennia. Therefore, the article will be in continuous citation and reference forever.

\section{Acknowledgments} 
The author thanks the anonymous reviewer for his/her several suggestions, which improved the article.

 \bibliographystyle{plain}
 \bibliography{reference.bib}

\begin{thebibliography}{10}

\bibitem{ANSTREICHER}
Kurt~M. Anstreicher.
\newblock The thirteen spheres: a new proof.
\newblock {\em Discrete Comput. Geom.}, 31(4):613--625, 2004.

\bibitem{BACHOCVELLENTIN}
Christine Bachoc and Frank Vallentin.
\newblock New upper bounds for kissing numbers from semidefinite programming.
\newblock {\em J. Amer. Math. Soc.}, 21(3):909--924, 2008.

\bibitem{BACHOCVALLENTIN2}
Christine Bachoc and Frank Vallentin.
\newblock Semidefinite programming, multivariate orthogonal polynomials, and
  codes in spherical caps.
\newblock {\em European J. Combin.}, 30(3):625--637, 2009.

\bibitem{BANNAIBANNAI}
Eiichi Bannai and Etsuko Bannai.
\newblock A survey on spherical designs and algebraic combinatorics on spheres.
\newblock {\em European J. Combin.}, 30(6):1392--1425, 2009.

\bibitem{BANNAISLOANE}
Eiichi Bannai and N.~J.~A. Sloane.
\newblock Uniqueness of certain spherical codes.
\newblock {\em Canadian J. Math.}, 33(2):437--449, 1981.

\bibitem{BARGMUSIN}
Alexander Barg and Oleg~R. Musin.
\newblock Codes in spherical caps.
\newblock {\em Adv. Math. Commun.}, 1(1):131--149, 2007.

\bibitem{BOROCZKY}
K\'{a}roly B\"{o}r\"{o}czky.
\newblock The {N}ewton-{G}regory problem revisited.
\newblock In {\em Discrete geometry}, volume 253, pages 103--110. Dekker, New
  York, 2003.

\bibitem{BOROCZKYGLAZYRIN}
K\'{a}roly~J. B\"{o}r\"{o}czky and Alexey Glazyrin.
\newblock Stability of optimal spherical codes.
\newblock {\em Monatsh. Math.}, 205(3):455--475, 2024.

\bibitem{BOYVALENKOVDRAGNEVHARDINSTOYANOVA2}
P.~G. Boyvalenkov, P.~D. Dragnev, D.~P. Hardin, E.~B. Saff, and M.~M.
  Stoyanova.
\newblock Universal lower bounds for potential energy of spherical codes.
\newblock {\em Constr. Approx.}, 44(3):385--415, 2016.

\bibitem{BOYVALENKOVDRAGNEVHARDINSTOYANOVA}
P.~G. Boyvalenkov, P.~D. Dragnev, D.~P. Hardin, E.~B. Saff, and M.~M.
  Stoyanova.
\newblock Bounds for spherical codes: the {L}evenshtein framework lifted.
\newblock {\em Math. Comp.}, 90(329):1323--1356, 2021.

\bibitem{BOYVALENKOVDANEV}
Peter Boyvalenkov and Danyo Danev.
\newblock On maximal codes in polynomial metric spaces.
\newblock In {\em Applied algebra, algebraic algorithms and error-correcting
  codes}, volume 1255 of {\em Lecture Notes in Comput. Sci.}, pages 29--38.
  Springer, Berlin, 1997.

\bibitem{BOYVALENKOVDANEVLANDGEV}
Peter Boyvalenkov, Danyo Danev, and Ivan Landgev.
\newblock On maximal spherical codes. {II}.
\newblock {\em J. Combin. Des.}, 7(5):316--326, 1999.

\bibitem{BOYVALENKOVDODUNEKOVMUSIN}
Peter Boyvalenkov, Stefan Dodunekov, and Oleg Musin.
\newblock A survey on the kissing numbers.
\newblock {\em Serdica Math. J.}, 38(4):507--522, 2012.

\bibitem{CASSELMAN}
Bill Casselman.
\newblock The difficulties of kissing in three dimensions.
\newblock {\em Notices Amer. Math. Soc.}, 51(8):884--885, 2004.

\bibitem{COHNJIAOKUMARTORQUATO}
Henry Cohn, Yang Jiao, Abhinav Kumar, and Salvatore Torquato.
\newblock Rigidity of spherical codes.
\newblock {\em Geom. Topol.}, 15(4):2235--2273, 2011.

\bibitem{COHNZHAO}
Henry Cohn and Yufei Zhao.
\newblock Sphere packing bounds via spherical codes.
\newblock {\em Duke Math. J.}, 163(10):1965--2002, 2014.

\bibitem{CONWAYSLOANE}
J.~H. Conway and N.~J.~A. Sloane.
\newblock {\em Sphere packings, lattices and groups}.
\newblock Springer-Verlag, New York, 1999.

\bibitem{DELSARTEGOETHALSSEIDEL}
P.~Delsarte, J.~M. Goethals, and J.~J. Seidel.
\newblock Spherical codes and designs.
\newblock {\em Geometriae Dedicata}, 6(3):363--388, 1977.

\bibitem{ERICSONZINOVIEV}
Thomas Ericson and Victor Zinoviev.
\newblock {\em Codes on {E}uclidean spheres}, volume~63 of {\em North-Holland
  Mathematical Library}.
\newblock North-Holland Publishing Co., Amsterdam, 2001.

\bibitem{GLAZYRIN}
Alexey Glazyrin.
\newblock A short solution of the kissing number problem in dimension three.
\newblock {\em Discrete Comput. Geom.}, 69(3):931--935, 2023.

\bibitem{JENSSENJOOSPERKINS}
Matthew Jenssen, Felix Joos, and Will Perkins.
\newblock On kissing numbers and spherical codes in high dimensions.
\newblock {\em Adv. Math.}, 335:307--321, 2018.

\bibitem{KALLALKANWANG}
Kenz Kallal, Tomoka Kan, and Eric Wang.
\newblock Improved lower bounds for kissing numbers in dimensions 25 through
  31.
\newblock {\em SIAM J. Discrete Math.}, 31(3):1895--1908, 2017.

\bibitem{KUKLIN}
N.~A. Kuklin.
\newblock Delsarte method in the problem on kissing numbers in high-dimensional
  spaces.
\newblock {\em Proc. Steklov Inst. Math.}, 284(1):S108--S123, 2014.

\bibitem{LEECH}
John Leech.
\newblock The problem of the thirteen spheres.
\newblock {\em Math. Gaz.}, 40:22--23, 1956.

\bibitem{LIBERTI}
Leo Liberti.
\newblock Mathematical programming bounds for kissing numbers.
\newblock In {\em Optimization and decision science: methodologies and
  applications}, volume 217 of {\em Springer Proc. Math. Stat.}, pages
  213--222. Springer, Cham, 2017.

\bibitem{MACHADOFILHO}
Fabr\'{\i}cio~Caluza Machado and Fernando~M\'{a}rio de~Oliveira~Filho.
\newblock Improving the semidefinite programming bound for the kissing number
  by exploiting polynomial symmetry.
\newblock {\em Exp. Math.}, 27(3):362--369, 2018.

\bibitem{MAEHARA}
H.~Maehara.
\newblock The problem of thirteen spheres---a proof for undergraduates.
\newblock {\em European J. Combin.}, 28(6):1770--1778, 2007.

\bibitem{MITTELMANNVALLENTIN}
Hans~D. Mittelmann and Frank Vallentin.
\newblock High-accuracy semidefinite programming bounds for kissing numbers.
\newblock {\em Experiment. Math.}, 19(2):175--179, 2010.

\bibitem{MUSIN3}
O.~R. Musin.
\newblock Bounds for codes by semidefinite programming.
\newblock {\em Tr. Mat. Inst. Steklova}, 263:143--158, 2008.

\bibitem{MUSIN1}
Oleg~R. Musin.
\newblock The kissing problem in three dimensions.
\newblock {\em Discrete Comput. Geom.}, 35(3):375--384, 2006.

\bibitem{MUSIN2}
Oleg~R. Musin.
\newblock The kissing number in four dimensions.
\newblock {\em Ann. of Math. (2)}, 168(1):1--32, 2008.

\bibitem{ODLYZKOSLOANE}
A.~M. Odlyzko and N.~J.~A. Sloane.
\newblock New bounds on the number of unit spheres that can touch a unit sphere
  in {$n$} dimensions.
\newblock {\em J. Combin. Theory Ser. A}, 26(2):210--214, 1979.

\bibitem{PFENDER}
Florian Pfender.
\newblock Improved {D}elsarte bounds for spherical codes in small dimensions.
\newblock {\em J. Combin. Theory Ser. A}, 114(6):1133--1147, 2007.

\bibitem{PFENDERZIEGLER}
Florian Pfender and G\"{u}nter~M. Ziegler.
\newblock Kissing numbers, sphere packings, and some unexpected proofs.
\newblock {\em Notices Amer. Math. Soc.}, 51(8):873--883, 2004.

\bibitem{SAMORODNITSKY}
Alex Samorodnitsky.
\newblock On linear programming bounds for spherical codes and designs.
\newblock {\em Discrete Comput. Geom.}, 31(3):385--394, 2004.

\bibitem{SARDARIZARGAR}
Naser~Talebizadeh Sardari and Masoud Zargar.
\newblock New upper bounds for spherical codes and packings.
\newblock {\em Math. Ann.}, 389(4):3653--3703, 2024.

\bibitem{SCHUTTEVANDERWAERDEN}
K.~Sch\"{u}tte and B.~L. van~der Waerden.
\newblock Das {P}roblem der dreizehn {K}ugeln.
\newblock {\em Math. Ann.}, 125:325--334, 1953.

\bibitem{SLOANE}
N.~J.~A. Sloane.
\newblock Tables of sphere packings and spherical codes.
\newblock {\em IEEE Trans. Inform. Theory}, 27(3):327--338, 1981.

\bibitem{WEGGEOLSEN}
Niels~Erik Wegge-Olsen.
\newblock {\em {{\(K\)}}-theory and {{\(C^*\)}}-algebras: a friendly approach}.
\newblock Oxford: Oxford University Press, 1993.

\bibitem{ZONG}
Chuanming Zong.
\newblock {\em Sphere packings}.
\newblock Universitext. Springer-Verlag, New York, 1999.

\end{thebibliography}

\end{document}